\documentclass[12pt,oneside,reqno]{amsart}
\usepackage{graphicx}
\usepackage{amsfonts,amssymb,amscd,amsmath,enumerate,verbatim,calc}
\usepackage[T1]{fontenc}
 \usepackage[utf8]{inputenc}
\usepackage{float}

\newtheorem{Theorem}{Theorem}[section]
\newtheorem{Lemma}[Theorem]{Lemma}
\newtheorem{Corollary}[Theorem]{Corollary}
\newtheorem{Proposition}[Theorem]{Proposition}
\newtheorem{Observation} [Theorem] {Observation}
\newtheorem{Remark}[Theorem]{Remark}

\begin{document}
\title{ON THE ZERO FORCING NUMBER OF CORONA and lexicographic product of GRAPHS}
\author{I. Javaid*, I. Irshad, M. Batool, Z. Raza }
\keywords{Zero forcing number, zero forcing sets, corona
product of graphs, lexicographic product of graphs. 
\indent $^*$ Corresponding author: imran.javaid@bzu.edu.pk}

\address{Centre for advanced studies in Pure and Applied Mathematics,
Bahauddin Zakariya University Multan, Pakistan\newline Email:
imran.javaid@bzu.edu.pk, iqrairshad9344@gmail.com,
m.batool12@yahoo.com, zahidsms@gmail.com}
\date{}
\maketitle
\begin{abstract} The zero forcing number of a graph $G$, denoted by $Z(G)$, is the minimum cardinality of a set $S$ of black
vertices (where vertices in $V(G)\setminus S$ are colored white)
such that $V(G)$ is turned black after finitely many applications of
$``$the color change rule$"$: a white vertex is turned black if it
is the only white neighbor of a black vertex. In this paper, we
study the zero forcing number of corona product, $G\odot H$ and
lexicographic product, $G\circ H$ of two graphs $G$ and $H$. It is
shown that if $G$ and $H$ are connected graphs of order $n_{1}\geq2$
and $n_{2}\geq2$ respectively, then $Z(G\odot ^{k}H)=Z(G\odot
^{k-1}H)+n_{1}(n_{2}+1)^{k-1}Z(H)$, where
$G\odot^{k}H=(G\odot^{k-1}H)\odot H$. Also, it is shown that for a
connected graph $G$ of order $n\geq 2$ and an arbitrary graph $H$
containing $l\geq 1$ components $H_{1},H_{2}, \cdots,H_{l}$ with
$|V(H_{i})|=m_{i}\geq 2$, $1\leq i\leq l$,
$(n-1)l+\sum\limits_{i=1}^l m_{i}\leq Z(G\circ H)\leq
n(\sum\limits_{i=1}^{l}m_{i})-l$.

\end{abstract}

 \textbf{\section{\textbf{Introduction}}}
 Let $G=(V(G),E(G))$
be a simple, undirected, connected graph with $|V(G)|\geq 2$. The
number of vertices and edges of $G$ are called the order and the
size of $G$ respectively. The degree of a vertex $v\in V$, denoted
by $deg_{G}(v)$, is the number of edges incident to the vertex $v$
in $G$. If there is no ambiguity, we will use the notation $deg(v)$
instead of $deg_{G}(v)$. An end vertex is a vertex of degree one.
Given $u,v\in V$, $u\sim v$ means that $u$ and $v$ are adjacent
vertices and $u\nsim v$ means that $u$ and $v$ are not adjacent. We
define the open neighborhood of a vertex $v$ in $G$,
$N_{G}(v)=\{u\in V(G):u\sim v\}$ and the closed neighborhood of $v$,
$N_{G}[v]=N_{G}(v)\cup\{v\}$. If there is no ambiguity, we will
simply write $N(v)$ or $N[v]$. If $u\in N_{G}(v)$ then $u$ is said
to be a neighbor of $v$. We denote a path, cycle, complete graph and
empty graph on $n$ vertices by $P_{n}$, $C_{n}$, $K_{n}$ and
$\overline{K_{n}}$ respectively. All graphs considered in this paper
are non trivial unless otherwise stated.
 \hfill \break
\indent The notion of a zero forcing set, as well as the associated
zero forcing number, of a simple graph was introduced in
\cite{BBBCCFGHHMNPS} to bound the minimum rank of associated
matrices for numerous families of graphs. Let each vertex of a graph
$G$ be given one of two colors, $``black"$ and $``white"$ by
convention. Let $S$ denote the initial set of black vertices of $G$.
The color-change rule converts the color of a vertex $u_{2}$ from
white to black if the white vertex $u_{2}$ is the only white
neighbor of a black vertex $u_{1}$; we say that $u_{1}$ forces
$u_{2}$, which we denote by $u_{1} \rightarrow u_{2}$. And a
sequence, $u_{1} \rightarrow u_{2}\rightarrow\cdots \rightarrow
u_{i}\rightarrow u_{i+1}\rightarrow \cdots \rightarrow u_{t},$
obtained through iterative applications of the color-change rule is
called a forcing chain. The set $S$ is said to be a zero forcing set
of $G$ if all the vertices of $G$ will be turned black after
finitely many applications of the color-change rule. The zero
forcing number of $G$, denoted by $Z(G)$, is the minimum of $|S|$
over all zero forcing sets $ S \subseteq V(G)$. A zero forcing set
of cardinality $Z(G)$ is called a forcing basis for $G$. For surveys
on the zero forcing parameter, see \cite{FH, FHR}. For more on the
zero forcing parameter in graphs, see \cite{BBFHHSDH, BFHRS, CDKY,
EHHLR, EKY, HHKMWY}.
\par If $F$ is a field, $M_{n}(F)$ denotes the set of all $n\times n$ matrices
over F. An $n$-square matrix $A$ is said to be a symmetric matrix if
$A^{T}=A$. The set of all real symmetric $n$-square matrices is
denoted by $S_{n}$. To a given graph $G$ with vertex set
$\{1,2,\cdots,n\}$, we associate a class of real, symmetric matrices
as follows:
$$S(G)=\{A=[a_{ij}]|A \in S_{n},\, \mbox{for}\,\,\, i\neq
j,a_{ij}\neq0\Leftrightarrow ij \in E(G)\}.$$ Note that there is no
restriction on the value of $a_{ii}$ with $i=1,2,\cdots,n$ and the
adjacency matrix $A(G)$ belongs to $S(G)$, where the adjacency
matrix of a graph $G$ is a square $(0,1)$-matrix of size $n$, whose
$(i,j)$-th entry is $1$ if and only if $v_{i}$ is adjacent to
$v_{j}$, since there are no loops in the graph, the diagonal entries
of the adjacency matrix are zero. On the other hand, the graph of an
$n$-square symmetric matrix $A$, denoted by $\mathcal{G}(A)$, is the
graph with vertices $\{1,2,\cdots,n\}$ and the edge set
$$\{ij|a_{ij}\neq0, 1\leq i\neq j\leq n\}.$$ The minimum rank of $G$ is defined to be
$$mr(G)=min\{rank(A)|A \in S(G)\},$$ while the maximum nullity of $G$ is defined
as $$M(G)=max\{null(A)|A \in S(G)\}.$$  We have
$$mr(G)+M(G)=|V(G)|.$$

 The underlying idea for the zero forcing set of a graph is that a black vertex is associated with a coordinate in a vector
that is required to be zero, while a white vertex indicates a
coordinate that can be either zero or nonzero. Changing a vertex
from white to black is essentially noting that the corresponding
coordinate is forced to be zero if the vector is in the kernel of a
matrix in $S(G)$ and all black vertices indicate coordinates assumed
to be or previously forced to be $0$. Hence the use of the term
``zero forcing set", see \cite{BBBCCFGHHMNPS}.

\par The support of a vector $x =(x_{i})$, denoted by $Supp(x)$, is
the set of indices $i$ such that $x_{i}\neq0$.
 Let $Z$ be a zero forcing set of $G$ and $A\in S(G)$. If $x\in
null(A)$ and
 $Supp(x)\cap Z=\phi $, then $x=0$, stated in \cite{BBBCCFGHHMNPS, YKML}.
Also from \cite{BBBCCFGHHMNPS, YKML}, we have $M(G)\leq Z(G)$ for a
graph $G$.
\par In this paper, we consider corona product and lexicographic product of graphs in the context of zero forcing
number. This paper consists of three sections. Section $1$ includes
introduction. Sections $2$ and $3$ include several results related
to the zero forcing number of corona and lexicographic product of
graphs, respectively.
\section{Corona Product of Graphs}
 \indent Let $G$ and $H$ be
two graphs of order $n_{1}$ and $n_{2}$ respectively. The corona
product of $G$ and $H$ is defined as the graph obtained from $G$ and
$H$ by taking one copy of $G$ and $n_{1}$ copies of $H$ and joining
by an edge each vertex from the $i^{th}$-copy of $H$ with the
$i^{th}$-vertex of $G$. We will denote by
$V=\{v_{1},v_{2},\cdots,v_{n_{1}}\}$, the set of vertices of $G$ and
by $H_{i}=(V_{i},E_{i})$, the $i$-th copy of $H$, where
$V_{i}=\{u^i_{1},u^i_{2},\cdots, u^i_{n_{2}}\}$, such that
$v_{i}\sim u^i_{k}$ for every $u^i_{k}\in V_{i}$. Note that the
subgraph of $G \odot H$ induced by $V_i$ is $H_i$ and the corona
graph $K_{1}\odot H$ is isomorphic to the join graph $K_{1}+H$. For
any integer $k\geq2$, we define the graph $G\odot^{k} H$ recursively
from $G\odot H$ as $G\odot^{k} H=(G\odot^{k-1} H)\odot H$. It is
also noted that $|G\odot ^{k-1}H|=n_{1}(n_{2}+1)^{^{k-1}}$ and
$|G\odot ^{k}H|=|G\odot ^{k-1}H|+n_{1}n_{2}(n_{2}+1)^{^{k-1}}.$
\par We call the copies of $H$ in $G\odot H$ as the copies of $H$ in
$1st$-corona, the newly added copies of $H$ in $G\odot H$ to obtain
$G\odot^{2} H$ as the copies of $H$ in $2nd$-corona and generally
the newly added copies of $H$ in $G\odot^{k-1} H$ to obtain
$G\odot^{k} H$ as the copies of $H$ in $k^{th}$-corona.

In $G\odot^{k} H$ for any positive integer $k$, we name the vertices
in $G\odot^{k-1} H$ as the root vertices of the copies of $H$ in
$k^{th}$-corona, that are joined to these vertices in $G\odot^{k}
H$. \par As one can color the vertices of $G\odot^{k} H$ in more
than one ways, but in this paper for a disconnected graph
$H$(containing isolated vertices) of order at least two, we will
consider the zero forcing set of $G\odot^{k} H$, that contains only
the vertices of $H$ but not the vertices of $G$.

In Figure \ref{new123}, the graph with grey vertices is $G=P_2$, the
copies of $H$ with black vertices are the copies of $H$ in first
corona and with white vertices are the copies of $H$ in 2nd corona.
The black and grey vertices are the root vertices of the
corresponding copies of $H$ with white vertices.
\begin{figure}[!ht]
   \centerline
     {\includegraphics[width=8cm]{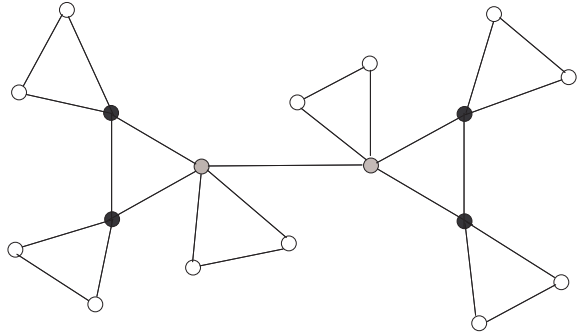}}
      \caption{$P_2\odot ^2 P_2$}\label{new123}
\end{figure}

We first recall the useful result obtained in \cite{EKE}.
\begin{Proposition}\label{manha} \cite{EKE} Let $G$ be a connected graph of order $n\geq
2$. Then
\\$(a)$ $Z(G)=1$ if and only if $G=P_{n},$
\\$(b)$ $Z(G)=n-1$ if and only if $G=K_{n}.$ \label{u}
\end{Proposition}
Note that for a connected graph $G$ of order $n$, we have
\begin{equation}
1\leq Z(G)\leq n-1.  \label{eqz}
\end{equation}%
\begin{Lemma}
Let $G$ be a connected graph of order $n \geq 2$ and let $H$ be a
graph of order at least two. Let $H_{i}$ be the subgraph of $G \odot
H$ corresponding to the $i^{th}$-copy of
$H$.\\
 $(i)$ If $S$ is a zero forcing set of $G \odot H$, then $V_{i} \cap S \neq
 \phi$ for every $i \in \{1,...,n\}.$\\
 $(ii)$ If $H$ is a connected graph and $S$ is a zero forcing set
 of $G \odot H$, then for every $i \in \{1,...,n\}$, $S \cap V_{i}$
 is a zero forcing set of $H_{i}$.
\end{Lemma}
\begin{proof}
 $(i)$ Suppose $V_{i}\cap S=\phi$, for some $i$. Then the vertex $v_{i}\in V$, initially black or forced to black in the zero forcing process,
  will not force any vertex in $V_{i}$ to turn black
because it has more than one white neighbors, a contradiction.
 \\$(ii)$ Suppose contrary that $S\cap V_{i}$ is not a zero forcing set for $H_{i}$.
 Then there exists a black vertex say $u^i_{l}\in V_{i}$, that has more
 than one white neighbors in graph $H_{i}$, so no forcing situation can occur. Note that the vertex $u^i_{l}\in V_{i}$ also has  more
 than one white neighbors in $G \odot H$, a
 contradiction.
\end{proof}
\begin{figure}[!ht]
      \centerline
      {\includegraphics[width=8cm]{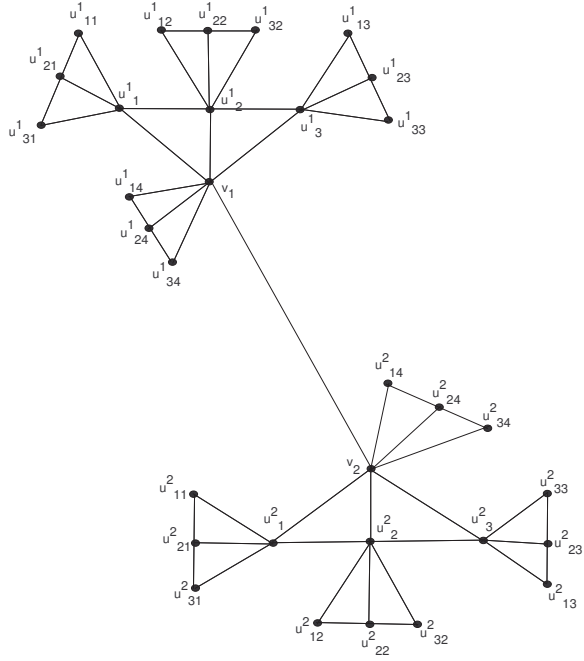}}
        \caption{Graph $P_{2}\odot^{2} P_{3}$}\label{must}
\end{figure}
Here we introduce some terminology related to the following theorem.
We assume that $B^{\prime} = \{v_{1},v_{2},...,v_{t}\}$ is a forcing
basis for $G$ and by using $B^{\prime}$ one can color all the other
vertices of $G$ by a sequence of forces in the following order:
$v_{t+1},v_{t+2},...,v_{n_{1}}$, with appropriate indexing of
vertices. We denote the vertices of $i$-th copies of $H$ in $l$-th
corona ($1\leq l\leq k$) by $u^i_{j_1j_2\cdots j_l}$, where $1\leq
j_1\leq n_2$ and $1\leq j_p\leq n_2+1$ for each $2\leq p\leq l$.
Assume that $Z(H)=m$ and a forcing basis for $i^{th}$-copy $H_{i}$
of $H$ in $1st$-corona is denoted by $ B^{i}=
\{u^i_{1},u^i_{2},\ldots,u^i_{m}\}$ and forcing basis for $i$-th
copies of $H$ in $l^{th}$-corona($2\leq l\leq k$) by $
B^i_{j_2\cdots j_l}=
\{u^i_{1j_2j_3...j_l},u^i_{2j_{2}j_{3}...j_{l}},\ldots,u^i_{mj_2j_3...j_l}\}$,
where $1\leq j_2, \cdots, j_l\leq n_{2}+1$. We denote the collection
of forcing basis of all copies of $H$ in first corona by $B_1$ i.e
$\cup_{i=1}^{n_1}B^i=B_1$, similarly the collection of forcing basis
of all copies of $H$ in $l$-th corona by $B_l$ i.e
$\cup_{i=1}^{^{n_{1}}}(\cup_{j_l=1}^{^{n_{2}+1}}\cdots
(\cup_{j_2=1}^{^{n_{2}+1}}B^i_{j_2 \cdots j_{l}}))=B_l$ and
$|B_l|=n_1(n_2+1)^{l-1}Z(H)$. FIGURE \ref{must} helps in
understanding the indices as mentioned.
\begin{Theorem} \label{api}
Let $G$ and $H$ be connected graphs of order $n_{1}$ and $n_{2}$
respectively, then $$Z(G\odot ^{k}H)=Z(G\odot
^{k-1}H)+n_{1}(n_{2}+1)^{k-1}Z(H).$$
\end{Theorem}
\begin{proof}
We prove the result by mathematical induction. \label{me}
For $k=1,$ we have to show that%
\begin{equation}
Z(G\odot H)=Z(G)+n_{1}Z(H).  \label{k=1}
\end{equation}%
\\First, we show that%
\begin{equation}
Z(G\odot H)\leq Z(G)+n_{1}Z(H).  \label{1}
\end{equation}%
We define $B = B^{\prime} \cup B_{1}$ and $|B| = Z(G) + n_{1}Z(H)$.
We claim that $B$ is a zero forcing set of $G \odot H$. To prove the
claim, first assume that $B$ is initially
 colored black and we color all the vertices of $H_{i}$ with
 $1 \leq i \leq t$ which are associated with the vertices of $B^{\prime}$
 using the corresponding sets $B^{i}$. Now all the vertices in $H_{i}$ associated with $v_{i}$, $1 \leq i \leq t$ are colored black.
 Note that there is a vertex $v_{i}$ belonging to $B^{\prime}$ that has
 only one white neighbor $v_{t+1}$. Thus $v_i\rightarrow v_{t+1}$. Then we color all the
 vertices in $H_{t+1}$ by using the black vertices in $B^{t+1}$.
 Continuing this process, we can color all the vertices of $G \odot
 H$.
\par Note that the degree of each vertex of $G$ is increased by
$n_{2}$ and the degree of each vertex of $H$ is increased by $1$ in
$G\odot H$. Let $v_{i}\in B^{\prime}$ and consider the corresponding
copy $H_{i}$ of $H$. Note that at least $Z(H)+1$ vertices are
required as initially colored black to start the zero forcing
process in each of these $H_{i}$'s, $1\leq i\leq t$. Then
$v_{t}\rightarrow v_{t+1}$ and to continue the process at least
$Z(H)$ more vertices are required in $H_{t+1}$ as initially colored
black. Continue the process until all the vertices are turned black.
Hence,
\begin{equation}
Z(G\odot H)\geq Z(G)+n_{1}Z(H).  \label{2}
\end{equation}%
By (\ref{1}) and (\ref{2}), (\ref{k=1}) holds.
\\Suppose that the result is true for $k-1$, i.e
\begin{equation}
Z(G\odot ^{k-1}H)=Z(G\odot ^{k-2}H)+n_{1}(n_{2}+1)^{k-2}Z(H).
\label{if}
\end{equation}%
Now we have to show that the result is true for $k$, i.e
\begin{equation}
Z(G\odot ^{k}H)=Z(G\odot ^{k-1}H)+n_{1}(n_{2}+1)^{k-1}Z(H).
\label{v}
\end{equation}%
We define $B^{\wp}=B^{\prime}\cup B_1\cup B_2\cup \cdots \cup
B_{(k-1)}$ and $|B^{\wp}|=Z(G)+n_{1}Z(H)+n_{1}(n_{2}+1)Z(H)+\cdots
+n_{1}(n_{2}+1)^{k-2}Z(H)=\alpha.$ Then $B^{\wp}$ is a zero forcing
set of $G\odot^{k-1} H$ by (\ref{if}). Therefore, we color all the
vertices of $G\odot^{k-1} H$ using $B^{\wp}$. Now all the vertices
from $1$st-corona to $(k-1)^{th}$-corona in $G\odot ^{k}H$ are
colored black. It suffices to show that
$n_{1}n_{2}(n_{2}+1)^{^{k-1}}$ vertices in the copies of $H$ in
$k^{th}$-corona in $G\odot ^{k}H$ will be colored black by taking
$n_{1}(n_{2}+1)^{^{k-1}}Z(H)$ more vertices as initially colored
black.
\\First, we show that
\begin{equation}
Z(G\odot ^{k}H)\leq Z(G\odot ^{k-1}H)+n_{1}(n_{2}+1)^{k-1}Z(H).
\label{m}
\end{equation}%
We define $B=B^{\wp}\cup B_k$ where
$|B|=\alpha+n_{1}(n_{2}+1)^{k-1}Z(H)$. We claim that $B$ is a zero
forcing set of $G\odot^{k} H$. Let $B$ is initially colored black.
Note that the degree of each vertex in $H_{i}$ in $k^{th}$-corona in
$G\odot ^{k}H$ is increased by one. We color all the vertices of
copies of $H$ in $k^{th}$-corona by using $B^i_{j_2\cdots j_k}$ and
the corresponding root vertex in $(k-1)^{th}$-corona. We obtain the
derived set of all black vertices in $G\odot ^{k}H$ resulting from
repeatedly applying the color-change rule. Hence, (\ref{m})
holds.\par Note that in each copy of $H$ in $k^{th}$-corona at least
$Z(H)$ vertices are required as initially colored black to continue
the zero forcing process.
 Hence,
 \begin{equation}
Z(G\odot ^{k}H)\geq Z(G\odot ^{k-1}H)+n_{1}(n_{2}+1)^{k-1}Z(H).
\label{e}
\end{equation}%
By (\ref{m}) and (\ref{e}), (\ref{v}) holds. Hence, the result is
true for any positive integer $k$.
\end{proof}
By using Theorem \ref{api} and Proposition \ref{manha}, we have the
following immediate corollaries:
\begin{Corollary}
Let $G$ and $H$ be connected graphs of order $n_{1}$, $n_{2} \geq 2$
respectively. Then $Z(G \odot^{k} H) = Z (G \odot^{k-1} H)  + n_{1}
(n_{2} + 1) ^{k-1}$ if and only if $H\cong P_{n_{2}}$.
\end{Corollary}
\begin{Corollary}
Let $G$ and $H$ be connected graphs of order $n_{1}$, $n_{2} \geq 2$
respectively. Then $Z(G \odot^{k} H) = Z (G \odot^{k-1} H)  + n_{1}
(n_{2} + 1) ^{k-1}(n_{2}-1)$ if and only if $H\cong K_{n_{2}}$.
\end{Corollary}
The wheel graph of order $n+1$ is defined as $W_{1,n}=K_{1}\odot
C_{n}$, where $K_{1}$ is the singleton graph. Any three pairwise
adjacent vertices of a wheel form a zero forcing set of $W_{1,n}$.
\begin{Remark} Let $W_{1,n}$, $n\geq3$, be a wheel graph. Then $Z(W_{1,n})=3.$
\end{Remark}
The fan graph $F_{n_{1},n_{2}}$ is defined as the join graph
$\overline{K_{n_{1}}}+P_{n_{2}}$. The case $n_{1}=1$ corresponds to
the usual fan graph $F_{1,n_{2}}$. Note that $F_{1,n_{2}}=K_{1}\odot
P_{n_{2}}$, where $K_{1}$ is the singleton graph. Two adjacent
vertices of $P_{n_{2}}$ where one must be an end vertex form a zero
forcing set of $F_{1,n_{2}}$.
\begin{Remark} Let $F_{1,n}$, $n\geq2$, be a fan graph. Then $Z(F_{1,n})=2.$
\end{Remark}

\begin{Theorem}
Let $G$ be a connected graph of order $n_{1}\geq 2$ and let $H$ be a
graph of order $n_{2}\geq 2$. Then $$Z(G\odot^{k} H)\leq
n_{1}(n_{2}+1)^{k-1}Z(K_{1}\odot H).$$
\end{Theorem}
 \begin{proof}
  We denote by $K_{1}\odot H_{i}$ the subgraph of $G\odot H$, obtained by joining
the vertex $v_{i}\in V$ with all the vertices of $H_{i}$. Let
$A_{i}$ be a forcing basis for $K_{1}\odot H_{i}$ and
$A=\cup_{i=1}^{n_{1}}A_{i}$ with $|A|=n_1Z(K_{1}\odot H)$. We show
that $A$ is a zero forcing set of $G\odot H$. Now there are two
cases:
\\Case $1$: $H$ contains no isolated vertex. We have that $v_{i}\in A_{i}$.
So $A$ contains all the vertices of $G$. Note that any black vertex
in $H_{i}$ has only one white neighbor, so after finite many
applications of the color-change rule all the vertices in $H_i$ for
each $i=1,2,\cdots, n_{1}$ are turned black. Now all the vertices in
$G\odot H$ are colored black.
\\Case $2$: $H$ contains isolated vertices, so there exists at least one vertex
$x\in H$ such that $x\nsim u$, for all $u\in H$. We have that
$v_{i}$ does not belong to the zero forcing set of minimum
cardinality of $K_{1}\odot H_{i}$. So $A$ does not contain any
vertex from $G$.
\\Subcase $2.1$: $H$ has only one isolated vertex. Let $x_{i}$ is the isolated
vertex of $H_{i}$ also $x_{i}\in A_{i}$ and $x_{i}\nsim u^{i}_p$ for
any $u^{i}_p\in V_{i}$. Thus $x_{i}\rightarrow v_{i}$, $1\leq i\leq
n_{1}$. Now all the vertices of $G$ are forced to black. Note that
any black vertex in $H_{i}$ has only one white neighbor, so after
finite many applications of the color-change rule all the vertices
in $H_i$ for each $i=1,2,\cdots, n_{1}$ are turned black. Now all
the vertices in $G\odot H$ are colored black.
\\Subcase $2.2$: $H$ has more than one isolated vertices, then all
isolated vertices in $H_{i}$ belong to $A_{i}$ except one, say
$y_{i}$ does not belong to $A_{i}$, then $v_{i}$ will be forced by
any isolated black vertex in $H_{i}$, $1\leq i\leq n_{1}$. Now all
the vertices of $G$ are forced to black and after finite iterative
applications of the color-change rule for connected subgraph of
$H_{i}$ all the vertices in these graphs are turned black for each
$i=1,2,\cdots, n_{1}$, and then $v_{i}\rightarrow y_{i}.$ Hence
$Z(G\odot H)\leq n_{1}Z(K_{1}\odot H).$ Therefore, the result
follows.
\end{proof}
\begin{Corollary} Let $G$ be a connected graph of order $n_{1}\geq
2$ and $H$ be a disconnected graph of order $n_{2}\geq 2$. Then
$$Z(G\odot^{k} H)=n_{1}(n_{2}+1)^{k-1}Z(K_{1}\odot H)=n_{1}(n_{2}+1)^{k-1}(n_{2}-1)$$ if and only if $H\cong
\overline{K_{n_{2}}}$. \label{q}
\end{Corollary}
\begin{proof}
 Suppose $H\cong \overline{K_{n_{2}}}$. For $k=1$, we have to show that $Z(G\odot H)=
n_{1}(n_{2} - 1)$. We define $B_{i} = V_{i} - \{u^i_{l}\}$, for any
$1 \leq l \leq n_{2}$, and for each $i=1,2,\cdots, n_{1}$ and
$B=\cup_{i=1}^{n_{1}}B_{i}$. We claim that $B$ is a zero forcing set
of $G\odot H$ with $|B|= n_{1}(n_{2} - 1) $. To prove the claim, we
first assume that $B$ is initially colored black. Note that every
initial black vertex of $H_{i}$ has single white neighbor $v_{i}$,
so $u^i_{1} \rightarrow v_{i}$. Now there is only one white vertex
$u^i_{l}$ in each $V_{i}$ and this vertex is the single white
neighbor of the corresponding vertex $v_{i}$, so $v_{i}\rightarrow
u^i_{l}$ for each $i=1,2,\cdots,n_{1}$.
\par Consider $v_{i}\in V$ and the corresponding copy $H_{i}$ of $H$
for any $i$, $1\leq i\leq n_{1}$. Note that at least $n_{2}-1$
vertices of $H_{i}$ are required to start the zero forcing process.
Hence, $Z(G\odot H)\geq n_{1}(n_{2}-1)$. Therefore, the result
follows. \par On the other hand, $Z(G\odot ^{k}H)=n_{1}(n_{2}-1)$
implies $H\cong \overline{K_{n_{2}}}$. Suppose $H\ncong
\overline{K_{n_{2}}}$ and let $u^i_{(n_{2}-2)}\sim u^i_{n_{2}}$, for
$1\leq i\leq n_{1}$. We define
$B_{i}=\{u^i_{1},u^i_{2},\cdots,u^i_{(n_{2}-1)}\}$, $1\leq i\leq
n_{1}-1$ and
$B_{n_{1}}=\{u^{n_1}_{1},u^{n_1}_{2},\cdots,u^{n_1}_{(n_{2}-2)}\}$.
Let $B=\cup_{i=1}^{n_{1}-1}B_{i}\cup B_{n_{1}}$ with
$|B|=n_{1}(n_{2}-1)-1$. We show that $B$ is a zero forcing set of
$G\odot H$. Assume that $B$ is initially colored black. Note that
each isolated vertex in each $H_{i}$, $1\leq i\leq n_{1}$, has only
one white neighbor $v_{i}\in V$, so $v_{i}$ will be forced to black
for each $i=1,2,\cdots,n_{1}$. Now all the vertices of $G$ are
colored black. Since $u^i_{(n_{2}-2)}$ has only one white neighbor
$u^i_{n_{2}}$ for each $i=1,2,\cdots,n_{1}$, so
$u^i_{(n_{2}-2)}\rightarrow u^i_{n_{2}}$, $1\leq i\leq n_{1}$. Note
that $v_{n_{1}}$ has only one white neighbor $u^{n_1}_{(n_{2}-1)}$,
so $v_{n_{1}}\rightarrow u^{n_1}_{(n_{2}-1)}$. Now all the vertices
are colored black. So $B$ is a zero forcing set of cardinality
$n_{1}(n_{2}-1)-1$, a contradiction. Therefore, the result follows.
\end{proof}
The following definitions are introduced in \cite{CDKY}. Fix a graph
$T$. A vertex of degree at least three is called a major vertex. An
end vertex $u$ is called a terminal vertex of a major vertex $v$ if
$d(u,v)<d(u,w)$ for every other major vertex $w$. The terminal
degree of a major vertex $v$ in $T$, denoted by $ter_{T}(v)$, is the
number of terminal vertices of $v$. A major vertex $v$ is an
exterior major vertex (emv) if it has positive terminal degree. Let
$\sigma(G)$ denote the sum of terminal degrees of all major vertices
of $G$ and let $ex(G)$ denote the number of emvs of $G$. We further
define an exterior degree two vertex to be a vertex of degree two
that lies on a path from a terminal vertex to its major vertex and
an interior degree two vertex to be a vertex of degree two such that
the shortest path to any terminal vertex includes a major vertex.

\begin{Theorem} \cite{CEJO,HM,S} If $T$ is a tree that is not a path, then
$dim(T)=\sigma(T)-ex(T)$. \label{n}
\end{Theorem}
\begin{Theorem}\cite{EKE} For any tree $T$, we have $Z(T)=dim(T)$ iff $T$ has
no interior degree two vertices and each major vertex $v$ of $T$
satisfies $ter_T(v)\geq2$. \label{r}
\end{Theorem}
\begin{Theorem}
Let $T$ be a tree of order $n\geq3$, that has no interior degree two
vertices and each major vertex $v$ of $T$ satisfies $ter_T(v)\geq2$,
then
\\$$ Z(T\odot^{k} K_{1}) =\left\{
  \begin{array}{ll}
    \sigma(T), & \hbox{$k=1$,} \\
    2^{k-2}n, & \hbox{$k\geq2$.}
  \end{array}
\right. $$\\
\end{Theorem}
\begin{proof}
 Since $T\odot^{k}K_{1}$ is a tree, with no interior degree two
vertex and each major vertex $v$ satisfies $ter_T(v)\geq2$. Now for
$k=1$, $\sigma(T\odot K_{1})=n$ and $ex(T\odot K_{1})=n-\sigma(T)$,
$dim(T\odot K_{1})=\sigma(T)=Z(T\odot K_{1})$ by Theorem \ref{n} and
Theorem \ref{r} we obtain the result. Since we have
$\sigma(T\odot^{2} K_{1})=2n$, $ex(T\odot^{2} K_{1})=n$, so we
obtain the result for $k=2$ by Theorem \ref{n} and Theorem \ref{r},
$dim(T\odot^{2} K_{1})=n=Z(T\odot^{2} K_{1})$.
\end{proof}
\par Let $\alpha$ be the
number of connected components of a graph $H$. Let us denote the
connected components of $H$ by $C_{l}$, where $1\leq l\leq\alpha$.
\begin{Theorem}
 Let $G$ be a connected graph of order $n_{1}$
and $H$ be a graph of order $n_{2}$. Let $\alpha$ be the number of
connected components of $H$ of order greater than one and let
$\beta$ be the number of isolated vertices of $H$. Then
\\ $Z(G\odot^{k} H) \leq\left\{
  \begin{array}{ll}
    n_{1}(n_{2}+1)^{k-1}\sum_{l=1}^{\alpha}Z(C_{l})+n_{1}(n_{2}+1)^{k-1}(\beta-1), & \hbox{$\alpha\geq1,\beta\geq2$,} \\
    Z(G\odot^{k-1}H)+n_{1}(n_{2}+1)^{k-1}\sum_{l=1}^{\alpha}Z(C_{l}), &
    \hbox{$\alpha\geq1,\beta=0$,} \\
    n_{1}(n_{2}+1)^{k-1}\sum_{l=1}^{\alpha}Z(C_{l})+n_{1}(n_{2}+1)^{k-1}-1,
    & \hbox{$\alpha\geq1,\beta=1$,} \\
    n_{1}(n_{2}+1)^{k-1}(n_{2}-1), & \hbox{$\alpha=0,\beta\geq2$.}
  \end{array}
\right. $\\
\end{Theorem}
\begin{proof}
 We define $K^i_{l}$, $1\leq l\leq\alpha$, be a forcing basis for connected component $C^i_{l}$ of $H_{i}$, $1\leq i\leq n_{1}$.

We suppose $\alpha\geq1, \beta\geq2$. We define $P_{i}$ to be the
set of vertices of $G\odot H$ formed by all but one of the isolated
vertices of $H_{i}$, $1\leq i\leq n_{1}$. Let us show that
$B=\cup_{i=1}^{n_{1}}(\cup_{l=1}^{\alpha}K^i_{l}\cup P_{i})$ is a
zero forcing set of $G\odot H$ with
$|B|=n_{1}\sum\limits_{l=1}^{\alpha}Z(C_{l})+n_{1}(\beta-1)$. Let
$B$ is initially colored black. Note that $p^i_{1}\in P_{i}$ has
only one white neighbor $v_{i}\in V$, so $p^i_{1}\rightarrow v_{i}$
for each $i$, $1\leq i\leq n_{1}$. Now all the vertices of $G$ are
colored black. We color all the vertices of connected components
$C^i_{l}$, $1\leq l\leq\alpha$, of $H_{i}$ using $K^i_{l}$ and the
corresponding vertex $v_{i}\in V$, $1\leq i\leq n_{1}$. Note that
the vertex $v_{i}$ has only one white neighbor $p^i_{\beta}$, the
isolated vertex of $H_{i}$ not belonging to $P_{i}$, so
$v_{i}\rightarrow p^i_{\beta}$, $1\leq i\leq n_{1}$, and we have the
derived set of all black vertices in $G\odot H$. As a consequence,
$Z(G\odot H)\leq
n_{1}\sum\limits_{l=1}^{\alpha}Z(C_{l})+n_{1}(\beta-1)$. Therefore,
the result follows. \par Now suppose $\alpha\geq1$, and $\beta=0$.
Let $B=B^{\prime}\cup_{i=1}^{n_{1}}(\cup_{l=1}^{\alpha}K^i_{l})$,
where $B^{\prime}=\{v_{1}, v_{2},\cdots,v_{t}\}$ is a forcing basis
for $G$ and by using $B^{\prime}$ one can color all the other
vertices of $G$ by a sequence of forces in the following order:
$v_{t+1}, v_{t+2},\cdots,v_{n}$, with appropriate indexing of
vertices. We show that $B$ is a zero forcing set of $G\odot H$.
Consider iterative applications of the color-change rule with
initial black set $B$. We color all the vertices of $H_{i}$ with $1
\leq i \leq t$ which are associated with the vertices of
$B^{\prime}$ using the corresponding sets $K^i_{l}$, $1\leq
l\leq\alpha$. Now all the vertices in $H_{i}$ associated with
$v_{i}$, $1 \leq i \leq t$ are colored black.
 Note that there is a vertex $v_{i}$ belonging to $B^{\prime}$ that has
 only one white neighbor $v_{t+1}\in V$.
 Thus $v_i\rightarrow v_{t+1}$. Then we color all the
 vertices in $H_{t+1}$ using the black vertices in $K^{t+1}_{l}$, $1\leq l\leq\alpha$.
 Continuing this process, we can color all the vertices of $G \odot
 H$. So $Z(G\odot H)\leq Z(G)+n_{1}\sum\limits_{l=1}^{\alpha}Z(C_{l})$.
 Therefore, the result follows. \par Now suppose $\alpha\geq1$, $\beta=1$. Let $R$ be the set of all
isolated vertices in each $H_{i}$, $1\leq i\leq n_{1}$, except one
say $r_{n_{1}}$. We define
$B=\cup_{i=1}^{n_{1}}(\cup_{l=1}^{\alpha}K^i_{l})\cup R$. We show
that $B$ is a zero forcing set of $G\odot H$. Let $B$ is initially
colored black. Note that the vertex $r_{i}\in R$, $1\leq i\leq
n_{1}-1$, has only one white neighbor $v_{i}\in V$, so
$r_{i}\rightarrow v_{i}$, $1\leq i\leq n_{1}-1$. We color all the
vertices of $H_{i}$,
 $1 \leq i \leq n_{1}-1$, which are associated with $v_{i}$, $1 \leq i \leq n_{1}-1$ using the corresponding
  sets $K^i_{l}$, $1\leq l\leq\alpha$. Now all the vertices in $H_{i}$ associated with $v_{i}$, $1 \leq i \leq n_{1}-1$ are colored
black. Now $v_{n_{1}-1}$ has only one white neighbor $v_{n_{1}}$, so
$v_{n_{1}-1}\rightarrow v_{n_{1}}$. We color
 $C^{n_{1}}_{l}$ in $H_{n_{1}}$ using $K^{n_{1}}_{l}$, $1\leq
 l\leq\alpha$. Now $v_{n_{1}}$ has only one white neighbor $r_{n_{1}}\in V_{n_{1}}$, so $v_{n_{1}}\rightarrow r_{n_{1}}$ and we have the
 derived set of all black vertices. So $Z(G\odot H)\leq
 n_{1}\sum\limits_{l=1}^{\alpha}Z(C_{l})+n_{1}-1$. Therefore, the result
 follows. \par Now suppose $\alpha=0$, $\beta\geq2$. Here $H$ is an empty graph so
 the result followed by Corollary \ref{q}.
\end{proof}
\section{Lexicographic Product of Graphs}
Let $G$ and $H$ be two graphs. The lexicographic product of $G$ and
$H$, denoted by $G\circ H$, is the graph with vertex set $V(G)\times
V(H)$ = $\{(a,v)\ |\ a\in V(G)$ and $v\in V(H)\}$, where $(a,v)$ is
adjacent to $(b,w)$ whenever $ab\in E(G)$ or $a=b$ and $vw\in E(H)$.
For any vertex $a\in V(G)$ and $b\in V(H)$, we define the vertex set
$H(a)$ = $\{(a,v)\in V(G\circ H)\ |\ v\in V(H)\}$ and $G(b)$ =
$\{(v,b)\in V(G\circ H)\ |\ v\in V(G)\}$. It is clear that the graph
induced by $H(a)$, called a layer $H(a)$, is isomorphic to $H$ and
the graph induced by $G(b)$, called a layer $G(b)$, is isomorphic to
$G$, denoted by $H(a)\cong H$ and $G(b)\cong G$ respectively. We
write $H(a)\sim H(b)$ when each vertex of $H(a)$ is adjacent to all
vertices of $H(b)$ and vice versa, and $H(a)\not \sim H(b)$ means
that no vertex of $H(a)$ is adjacent to any vertex of $H(b)$ and
vice versa.
\par Let $G$ be a connected graph and $H$ be a non-trivial
graph containing $k\geq 1$ components $H_{1},H_{2},\cdots,H_{k}$
with $|V(H_{j})|\geq 2$ for each $j=1,2,\ldots,k$. For any vertex
$a\in V(G)$ and $1\leq i\leq k$, we define the vertex set $H_{i}(a)$
= $\{(a,v)\in V(G\circ H)\ | v\in V(H_{i})\}$. Let
$|V(H_{i})|=m_{i}$, $1\leq i\leq k$. From the definition of $G\circ
H$, it is clear that for every $(a, v)\in V(G\circ H)$, $deg_{G\circ
H}(a, v)= deg_{G}(a)\cdot |V(H)|+ deg_{H}(v)$. If $G$ is a
disconnected graph having $k\geq 2$ components $G_1$, $G_2$, \ldots,
$G_k$, then $G\circ H$ is also a disconnected graph having $k$
components such that $G\circ H=G_1 \circ H\cup G_2 \circ H\cup
\ldots \cup G_k\circ H$ and each component $G_i\circ H$ is the
lexicographic product of connected component $G_i$ of $G$ with $H$,
therefore throughout this section, we will assume $G$ to be
connected.
\begin{Observation}\label{obs1} For any $a,b\in V(G)$, either $H(a)\sim H(b)$ or
$H(a)\not \sim H(b)$ in $G\circ H$.
\end{Observation}
First, we give a general lower bound on the zero forcing number of
lexicographic product of graphs. Note that given any connected graph
$G$, then $Z(G)=1$ if and only if $G\cong P_{n}$, $n\geq 2$. So, if
$Z(G \circ H)=1$ for some graph $H$, then clearly $G \circ H$ is a
path graph, i.e $G$ is the trivial graph $K_{1}$ and $H$ is a path
or viceversa. So, we have the following result:
\begin{Remark} If $G$ and $H$ are non trivial graphs, then $Z(G\circ H)\geq
2$.
\end{Remark}
\begin{Lemma}\label{lem 1} Let $G$ be a connected graph on $n$ vertices. There exists a forcing basis $S$
for $G+K_{1}$ such that $S\subseteq V(G)$.
\end{Lemma}
\begin{proof} Let $V(G+K_{1})=V(G)\cup \{v\}$. If $v\notin S$ we
have  nothing to prove. Suppose that $v\in S$. Since $G$ is
connected and $deg_{G+K_1}(v)=n$ and also $v$ is initially colored
black so by equation (\ref{eqz}), there exists at least one white
vertex $x\in N_{G+K_1}(v)$ such that $(S\setminus \{v\})\cup \{x\}$
is a forcing basis for $G+K_{1}$.
\end{proof}
\begin{Theorem} \label{thm1}
Let $G$ be a connected graph and $H$ be an arbitrary graph
containing $k\geq 1$ components $H_{1},H_{2},H_{3},\ldots H_{k}$ and
$m_i\geq 2$. Let $Z$ be a zero forcing set of $G\circ H$. For any
vertex $a\in V(G)$, if $Z_{i}(a)= Z\cap H_{i}(a)$ for every $i \in
\{1,2,\ldots k\}$, then $Z_{i}(a)\neq \phi$. Moreover, if $B_{i}$ is
a forcing basis for $H_{i}$, then $|Z_{i}(a)|\geq |B_{i}|$.
\end{Theorem}
\begin{proof}
Suppose that for some $i\in \{1,2,\ldots,k\}$ there exists a vertex
$a\in V(G)$ such that $Z_{i}(a)=\phi$. Then, by Observation
\ref{obs1} any vertex in $H_{i}(a)$ cannot be forced by any vertex
in $H_{j}(b)$, $i\neq j$ for any $a\neq b\in V(G)$, a contradiction.
\par Now suppose that $|Z_{i}(a)|<|B_{i}|$ and
$Z_{i}(a)=\{(a,z_{1}),(a,z_{2}),\ldots,(a,z_{t})\}$ for some forcing
basis $B_{i}$ of $H_{i}$, where $\{z_{1},z_{2},\ldots,z_{t}\}\subset
V(H_{i})$. Then, each black vertex in $H_{i}(a)$ has more than one
white neighbors and no vertex of $H_{i}(a)$ can be forced by any
vertex in $H_{j}(v)$ for any $v\in V(G)$, $i\neq j$. Hence,
$|Z_{i}(a)|\geq |B_{i}|$.
\end{proof}
From above theorem, we have an immediate corollary:
\begin{Corollary}
Let $G$ be a connected graph and $H$ be an arbitrary  graph
containing $k\geq 1$ components $H_{1},H_{2},H_{3},\ldots H_{k}$ and
$m_i\geq 2$. Let $Z(a)=\bigcup_{1\leq i\leq k} Z_{i}(a)$ for $a\in
V(G)$. Then $Z(a)$ is a zero forcing set of $H(a)$.
\end{Corollary}
\begin{Proposition} Let $G$ be a connected graph and $H$
be an arbitrary graph containing $k\geq1$ components
$H_{1},H_{2},\cdots,H_{k}$ and $m_i\geq2$. Let $a\in V(G)$ and $Z$
be a forcing basis for $G\circ H$. If $Z(a)=Z\cap H(a)$ and
$\alpha(a)=|Z(a)|$. Then $$\alpha(a)\leq \sum_{i=1}^{k}m_{i}.$$
\end{Proposition}
\begin{proof} For any $(b,x)\in V(G\circ H)$, $deg_{G\circ H}(b,x)=\sum\limits_{b\sim
u}|H(u)|+deg_{H}(x)$. Since $G$ is connected so for any $b\in V(G)$,
there exist at least one vertex $v\in V(G)$ such that $b\sim v$ and
$deg_{G\circ H}(b,x)\geq |H(v)|+deg_{H}(x)$. To start the zero
forcing process at least all the vertices of $H(v)$ along with
$deg_{H}(x)$ vertices are initially colored black. Since $Z$ is a
forcing basis for $G\circ H$. Hence, $Z\cap H(v)=H(v)$ and
$\alpha(a)\leq \sum\limits_{i=1}^{k}m_{i}$ for any $a\in V(G)$.
\end{proof}
\begin{Corollary} \label{cor1} Let $G$ be a connected graph and $H$
be an arbitrary graph containing $k\geq1$ components
$H_{1},H_{2},\cdots,H_{k}$ and $m_i\geq2$. Then there exists at
least one vertex $x\in V(G)$ such that
$\alpha(x)=\sum\limits_{i=1}^{k}m_{i}$.
\end{Corollary}
The projection of $S\subseteq V(G \circ H)$ onto $G$, denoted by
$P_{G}(S)$, is the set of vertices $a\in V(G)$ for which there
exists a vertex $(a,v)\in S$. Similarly, the projection of
$S\subseteq V(G \circ H)$ onto $H$, $P_{H}(S)$, is the set of
vertices $v\in V(H)$ for which there exists a vertex $(a,v)\in S$.
\begin{Lemma}\label{lem 3} Let $G$ be a connected graph of order $n$ and $H$
be an arbitrary graph containing $k\geq1$ components
$H_{1},H_{2},\cdots,H_{k}$ and $m_i\geq2$. Let $Z$ be a forcing
basis for $G \circ H$ and $Z_{i}=Z\cap V(G\circ H_{i})$, where
$G\circ H_{i}$ is the induced subgraph of $G \circ H$. Then
$P_{G}(Z_{i})=V(G)$.
\end{Lemma}
\begin{proof} Let $V(G)=\{u_{1},u_{2},\cdots,u_{n}\}$ and
$V(H_{i})=\{v_{1}^{i},v_{2}^{i},\cdots,v_{m_{i}}^{i}\}$ for $1\leq
i\leq k$. Suppose $P_{G}(Z_{i})\neq V(G)$, i.e there exists a vertex
$u_{j}\in V(G)$ such that $u_{j}\not \in P_{G}(Z_{i})$. This implies
$(u_{j},v_{p}^{i})\not \in Z$ for any $v_{p}^{i}\in V(H_{i})$ for
$1\leq p\leq m_{i}$. Hence, $H_{i}(u_{j})\cap Z=\phi$, a
contradiction by Theorem \ref{thm1}.
\end{proof}

\begin{Lemma} Let $G$ be a connected graph of order $n$ and $H$
be an arbitrary graph containing $k\geq1$ components
$H_{1},H_{2},\cdots,H_{k}$ and $m_i\geq2$. Then
$$Z(G\circ H)\leq n(\sum _{i=1}^{k}m_{i})-k.$$
\end{Lemma}
\begin{proof} Let $V(G)=\{u_{1},u_{2},\cdots,u_{n}\}$ and
$V(H_{i})=\{v_{1}^{i},v_{2}^{i},\cdots,v_{m_{i}}^{i}\}$ for $1\leq
i\leq k$. We define $Z=V(G\circ H)\setminus\{(u_{1},v_{2}^{i})|1\leq
i\leq k\}$ and $|Z|=n(\sum\limits_{i=1}^{k}m_{i})-k$. We claim that
$Z$ is a zero forcing set of $G\circ H$. To prove the claim, assume
that $Z$ is initially colored black. Since for any $i\neq j$,
$v_{p}^{i}\not \sim v_{q}^{j}$ in $H$ hence $(u_{r},v_{p}^{i})\not
\sim (u_{r},v_{q}^{j})$, $1\leq r\leq n$, in $G\circ H$. Since for
any $i$, $H_{i}$ is connected, so there exists at least one vertex
$v_{l}^{i}$ such that $v_{2}^{i}\sim v_{l}^{i}$ in $H_i$ and hence
$(u_{r},v_{2}^{i})\sim (u_{r},v_{l}^{i})$, $1\leq r\leq n$, in
$G\circ H$. Therefore,
$(u_{1},v_{l}^{i})\rightarrow(u_{1},v_{2}^{i})$ for $1\leq r\leq n$.
Hence, $Z(G\circ H)\leq n(\sum\limits _{i=1}^{k}m_{i})-k.$
\end{proof}
This bound is sharp for $G=K_{n}$ and
$H=K_{m_{1}},K_{m_{2}},\cdots,K_{m_{k}}.$
\begin{Lemma} Let $G$ be a connected graph of order $n$ and $H$
be an arbitrary graph containing $k\geq1$ components
$H_{1},H_{2},\cdots,H_{k}$ and $m_i\geq2$. Then
$$Z(G\circ H)\geq (n-1)k+\sum _{i=1}^{k}m_{i}.$$
\end{Lemma}
\begin{proof} The result follows from Corollary \ref{cor1} and
Lemma \ref{lem 3}.
\end{proof}
This bound is sharp for $G=K_{1,n-1}$ and
$H=P_{m_{1}},P_{m_{2}},\cdots,P_{m_{k}}$.
\begin{Lemma} Let $G$ be a connected graph of order $n$ and $H_{1}, H_{2},\cdots,
H_{k}$, $k\geq 2$ are singleton components. Then $Z(G\circ H)\leq
nk-2$.
\end{Lemma}
\begin{proof}Let $V(G)=\{u_{1},u_{2},\cdots,u_{n}\}$ and
$V(H_{i})=\{x_{i}\}$ for $1\leq i\leq k$. We define, for $u_{1}\sim
u_{2}$, $Z=V(G\circ H)\setminus\{(u_{1},x_{k}),(u_{2},x_{k})\}$ with
$|Z|=nk-2$. Since $x_{1} \not \sim x_{k}$ so $(u_{1},x_{1})
\rightarrow (u_{2},x_{k})$. Similarly, $(u_{2},x_{1})\rightarrow
(u_{1},x_{k})$. Now all the vertices in $G\circ H$ are colored
black. Therefore, the result follows.
\end{proof}
Now we study the zero forcing number of lexicographic product of
graphs for some specific families of graphs and $H$ contains one
component only. Note that $K_m\circ K_n\cong K_{mn}$, so $Z(K_M\circ
K_n)=mn-1$. Therefore, from now on we consider the graphs when at
most one of the factors of the product is a complete graph.
\begin{Lemma}For any connected graph $H$ of order $m$, $Z(K_n\circ
H)=Z(H)+(n-1)m$.
\end{Lemma}
\begin{proof} Note that for any $a\in V(K_n)$, $H(a)\sim H(b)$ for all $b\in
V(K_n)\setminus\{a\}$. Therefore all the vertices in $H(b)$ for all
$b\in V(K_n)\setminus\{a\}$ are initially colored black. Now
$H(a)\cong H$ so $Z(H)$ vertices are required as initially colored
black to complete the zero forcing process.
\end{proof}
Since the lexicographic product of graphs is not commutative, i.e
$K_n\circ H \not \cong H\circ K_n$. Therefore we study the case when
the second factor is a complete graph.
\begin{Lemma} For any connected non complete graph $G$ of order $m$, $m(n-1)+1\leq Z(G\circ K_n)\leq
nm-2$.
\end{Lemma}
\begin{proof}
Suppose $V(G)=\{u_1,u_2, \cdots, u_m\}$ and $V(K_n)=\{v_1,v_2,
\cdots, v_n\}$. It is easy to check that $Z=V(G\circ
K_n)\setminus\{(u_1,v_2), (u_n,v_2)\}$ is a zero forcing set of
$G\circ K_n$. Hence $Z(G\circ K_n)\leq nm-2$. \par By Corollary
\ref{cor1}, there exists at least one vertex $a\in V(G)$ such that
$\alpha(a)=n$. Since $H(v)\cong K_n$, therefore for $m-1$ layers
$H(v)$ at least $n-1$ vertices from each layer are required as
initially colored black to color all the vertices of $G\circ K_n$.
Hence $Z(G\circ K_n)\geq n+(m-1)(n-1)=m(n-1)+1$.
\end{proof}
Now we study the zero forcing number of $P_n\circ H$, for $n\geq3$
and a connected graph $H$. Suppose $V(P_n)=\{u_1, u_2,\cdots,
u_n\}$. Since $H(u_i)\sim H(u_{i+1})$, $1\leq i\leq n-1$. Note that
to color the vertices of $H(u_1)$, $Z(H)$ vertices in $H(u_1)$ and
all the vertices in $H(u_2)$ are required as initially colored black
and to color the vertices of $H(v_3)$, $Z(H)$ vertices in $H(v_3)$
and all the vertices in $H(v_4)$ are required as initially colored
black. Continue the process untill all the vertices in $P_n\circ H$
are turned black.
\begin{Proposition} For a connected graph $H$ and $n\geq3$, $$Z(P_n\circ H)=\left\{
  \begin{array}{ll}
    \frac{n(Z(H)+m)}{2}, & \hbox{$n$\mbox{ is even}} \\
   \frac{n(Z(H)+m)+Z(H)-m}{2}, & \hbox{$n$\mbox{ is odd}.}
  \end{array}
\right.$$
\end{Proposition}
\begin{Corollary} For $n,m\geq 3$,
$$Z(P_n\circ K_m)=\left\{
  \begin{array}{ll}
    \frac{n(2m-1)}{2}, & \hbox{$n$\mbox{ is even}} \\
   nm-\frac{n+1}{2}, & \hbox{$n$\mbox{ is odd}.}
  \end{array}
\right.$$
\end{Corollary}
Now we study the zero forcing number of $C_n\circ H$, for $n\geq4$
and a connected graph $H$. Suppose $V(C_n)=\{u_1, u_2,\cdots,
u_n\}$. Since $H(u_i)\sim H(u_{i+1})$, $1\leq i\leq n$ and
$u_{n+1}=u_1$. Note that to color the vertices of $H(u_1)$, $Z(H)$
vertices in $H(u_1)$ and all the vertices in $H(u_2)$ and $H(u_n)$
are required as initially colored black and to color the vertices of
$H(v_3)$, $Z(H)$ vertices in $H(v_3)$ and all the vertices in
$H(v_4)$ are required as initially colored black. Continue the
process untill all the vertices in $C_n\circ H$ are turned black.
\begin{Proposition} For a connected graph $H$ and $n\geq4$,
$$Z(C_n\circ H)=\left\{\begin{array}{ll}
    \frac{n(Z(H)+m)}{2}, & \hbox{$n$\mbox{ is even}} \\
   \frac{n(m+Z(H))+m-Z(H)}{2}, & \hbox{$n$\mbox{ is odd}.}
  \end{array}
\right.$$
\end{Proposition}
\begin{Corollary} For $n\geq4$ and $m\geq 3$,
$$Z(C_n\circ K_m)=\left\{\begin{array}{ll}
    \frac{n(2m-1)}{2}, & \hbox{$n$\mbox{ is even}} \\
   \frac{n(2m-1)+1}{2}, & \hbox{$n$\mbox{ is odd}.}
  \end{array}
\right.$$
\end{Corollary}

\end{document}